\documentclass[11pt]{amsart}
\usepackage{graphicx}
\usepackage[usenames]{color}
\usepackage{amssymb,amscd}
\newcommand{\bea}{\begin{eqnarray*}}
\newcommand{\eea}{\end{eqnarray*}}

\newcommand{\bm}{\begin{pmatrix}}
\newcommand{\fm}{\end{pmatrix}}

\newcommand\Z{\mathbb Z}
\newcommand\N{\mathbb N}
\newcommand\Q{\mathbb Q}
\newcommand\R{\mathbb R}

\renewcommand\L{\mathcal L}

\newcommand\GGG{\Gamma_d(q)}
\newcommand\Gk{\Gamma^k_d(q)}

\newtheorem{theorem}{Theorem}[section]
\newtheorem{proposition}[theorem]{Proposition}
\newtheorem{prop}[theorem]{Proposition}
\newtheorem{lemma}[theorem]{Lemma}
\newtheorem{observation}[theorem]{Observation}
\newtheorem{corollary}[theorem]{Corollary}
\newtheorem{definition}[theorem]{Definition}

\title{Bilipschitz versus quasi-isometric equivalence for higher rank lamplighter groups}
\author{Tullia Dymarz}
\address{Department of Mathematics,
University of Wisconsin-Madison, 480 Lincoln Drive,  Madison, WI 53706} \email{dymarz@math.wisc.edu}
\author{Irine Peng}
\address{Department of Mathematics, Pohang University of Science and Technology, 31 San, Hyoja-dong, Namgu, Pohang, Gyeongsanbukdo, South Korea.}
\email{irinepeng@postech.ac.kr}\author{Jennifer Taback}
\address{Department of Mathematics,
Bowdoin College, 8600 College Station, Brunswick, ME 04011} \email{jtaback@bowdoin.edu}
\thanks{The first author acknowledges support from National Science Foundation grant DMS-1207296. The third author acknowledges support from National Science Foundation grant DMS-1105407.}
\keywords{quasi-isometric equivalence, bilipschitz equivalence, higher rank lamplighter groups, Diestel-Leader graphs, Diestel-Leader groups}
\begin{document}
\maketitle

\begin{abstract}
We describe a family of finitely presented groups which are quasi-isometric but not bilipschitz equivalent.  The first such examples were described by the first author in \cite{Dy} and are the lamplighter groups $F \wr \Z$ where $F$ is a finite group; these groups are finitely generated but not finitely presented.   The examples presented in this paper are higher rank generalizations of these lamplighter groups and include groups that are of type $F_n$ for any $n$.\end{abstract}

\section{Introduction}
In this paper we present the first finitely presented examples of groups which are quasi-isometric but not bilipschitz equivalent, generalizing the finitely generated examples given by the first author in \cite{Dy}.  Moreover, for any $n$ this family of examples contains groups which are of type $F_n$.

The groups used in \cite{Dy} to construct finitely generated examples of groups with this property were the lamplighter groups $F \wr \Z$, where $F$ is a finite group; our examples are higher rank analogues of these groups.  The main theorem of \cite{Dy} relies on the fact that with respect to a certain generating set, the Cayley graph of $F \wr \Z$ is a {\em Diestel-Leader graph}, which is defined as a particular subset of the product of two trees.  The higher rank lamplighter groups have a preferred generating set with resulting Cayley graph identified with the one skeleton of ``larger" Diestel-Leader complexes which are subsets of products of more than two trees.
We denote these groups $\GGG$ and refer to them as {\em Diestel-Leader groups}; the corresponding Cayley graphs are denoted $DL_d(q)$.
 These graphs and their geometry are discussed in Section \ref{sec:groups} below.

The two results are as follows:
\begin{theorem}\label{thebigone}[\cite{Dy}, Theorem 1.1]
Let $F$ and $G$ be finite groups with $|F|=n$ and $|G|=n^k$ where $k>1$. Then there does not exist a bilipschitz equivalence between the lamplighter groups $G \wr \Z$ and $F \wr \Z$  if $k$ is not a product of prime factors appearing in $n$.
\end{theorem}

The proof of this theorem and of Theorem \ref{thm:main}  and Corollary \ref{cor:mainthm} below rely on the existence of an index $k$ subgroup of the original group, which is necessarily quasi-isometric to the group; the Cayley graph of this finite index subgroup is the graph $DL_d^k(q)$ (see Section \ref{finindsec} for a definition).  Our main result below concerns these graphs, which we show to be quasi-isometric but not bilipschitz equivalent to the original Cayley graphs $DL_d(q)$ for all parameter values $d$ and $q$.
Only for certain values of these parameters is there a corresponding Diestel-Leader group $\Gamma_d(q)$ whose Cayley graph (relative to a given generating set) is $DL_d(q)$.  In those cases, we obtain the analogous result on the level of groups.

\begin{theorem}\label{thm:main} $DL_d^k(q)$ is quasi-isometric to $DL_d(q)$ but not bilipschitz equivalent if $k$ is not a product of prime factors appearing in $q$. \end{theorem}

\begin{corollary} \label{cor:mainthm} For each $n$ there exist groups of type $F_n$ which are quasi-isometric but not bilipschitz equivalent.
\end{corollary}

The notions of quasi-isometric equivalence and bilipschitz equivalence of metric spaces are closely related; quasi-isometric equivalence is a coarse generalization of bilipschitz equivalence, in the following sense.  Let $X$ and $Y$ be metric spaces with metrics $d_X$ and $d_Y$ respectively.

1. A $K$-{\em bilipschitz equivalence} $g: X \rightarrow Y$ is a bijection satisfying, for all $x,y \in X$,
$$\frac{1}{K}d_X(x,y) \leq d_Y(f(x),f(y)) \leq K d_X(x,y).$$

2. A $(K,C)$-{\em quasi-isometric equivalence} $f: X \rightarrow Y$ is a map satisfying, for all $x,y \in X$,
\begin{enumerate}
\item $\frac{1}{K}d_X(x,y)-C \leq d_Y(f(x),f(y)) \leq K d_X(x,y)+C$, and

\medskip

\item $Nbhd_C(f(X)) = Y$.
\end{enumerate}
For discrete groups, a bilipschitz map is equivalent to a bijective quasi-isometry.
A natural question to ask is for which classes of metric spaces these two notions coincide.  We are further interested in this question for finitely generated groups, which are considered as metric spaces with the word metric $d_S$ arising from a finite generating set $S$.  These notions of equivalence arise naturally for finitely generated groups.  If $S_1$ and $S_2$ are two finite generating sets for a group $G$, then the resulting Cayley graphs are bilipschitz equivalent as metric spaces with the corresponding word metrics.

Earlier examples of metric spaces for which quasi-isometric and bilipschitz equivalence are distinct were given by Burago-Kleiner \cite{BK1} and McMullen \cite{M}.  Both exhibit separated nets in $\R^2$ which are quasi-isometric but not bilipschitz equivalent.  However, these examples do not correspond to the Cayley graphs of finitely generated groups.

This question for finitely generated groups was previously studied by Whyte in \cite{W}, who finds a possible obstruction to their equivalence when the groups are amenable.   Whyte developed a criterion using uniformly finite homology to determine when a map between certain geometric spaces is a bounded distance from a bijection.  We use his results below, but do not develop the theory of uniformly finite homology here; we refer the reader to \cite{BW1,BW2,Dy} for more details.

The proofs of Theorems \ref{thebigone} and \ref{thm:main} are parallel in many ways; we strive to highlight the intricate geometry of the Diestel-Leader groups in this paper, and quote results from \cite{Dy} which are unchanged between the two contexts.  We refer the reader to \cite{Dy} for any omitted proofs.

The first author would like to thank Kevin Wortman for useful conversations.

\section{Groups and geometry}\label{sec:groups}
The proof of Theorem \ref{thebigone} in \cite{Dy} relies on the fact that the geometry of the lamplighter group $F \wr \Z$, where $F$ is a finite group of order $n$, is identified with a Diestel-Leader graph, which is a certain subspace of a product of two trees and defined below in Section \ref{sec:DLgraph}.  That is, there is a particular generating set with respect to which the Cayley graph for $F \wr \Z$ is exactly this Diestel-Leader graph.

The construction of a Diestel-Leader graph is more robust, and one can define analogous graphs which are subsets of the product of any number of trees.  The groups we consider below possess finite generating sets for which the resulting Cayley graphs can be identified with the $1$-skeleton of a ``larger" Diestel-Leader complex.  We make this precise below.

\subsection{Diestel-Leader graphs and complexes.} \label{sec:DLgraph} Let  $T^{q+1}$ denote the infinite regular $q+1$ valent tree oriented with $q$ incoming edges and $1$ outgoing edge at each vertex. Fix a base point and identify each edge with the unit interval. This yields a height function $h:T^{q+1} \rightarrow \R$ on the tree that sends vertices surjectively to $\Z$ and maps the base point to zero.  To be consistent with \cite{BNW} we orient the tree so that the height decreases across any incoming edge.

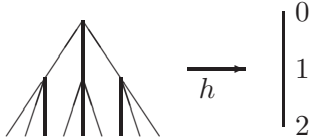
\begin{figure}[h]
\label{fig:trees}
\setlength{\unitlength}{.1in}
\center
\begin{picture}(60,8)(1,.5)
\linethickness{1pt}
\newsavebox{\treeones}
\savebox{\treeones}(-4,8){
\put(-4,1){\line(2,3){4}}
\put(-3,1){\line(1,2){3}}
\put(-2,1){\line(1,3){2}}

\put(-1,1){\line(1,6){1}}
\put(0,1){\line(0,6){6}}
\put(1,1){\line(-1,6){1}}

\put(2,1){\line(-1,3){2}}
\put(3,1){\line(-1,2){3}}
\put(4,1){\line(-2,3){4}}
}

%
%

\newsavebox{\treetwos}
\savebox{\treetwos}(11,8){
\put(11,1){\line(2,3){4}}
\put(12,1){\line(1,3){1}}
\put(13,1){\line(0,1){3}}

\put(14,1){\line(1,3){1}}
\put(15,1){\line(0,6){6}}
\put(16,1){\line(-1,3){1}}

\put(17,1){\line(0,1){3}}
\put(18,1){\line(-1,3){1}}
\put(19,1){\line(-2,3){4}}
}

\put(4,-1){\usebox{\treetwos}}

%
%
%
\put(30,4){\vector(1,0){3}}
\put(35,1){\line(0,1){6}}
\put(31,3){\makebox(0,0){$h$}}
\put(36,1){\makebox(0,0){$2$}}
\put(36,4){\makebox(0,0){$1$}}
\put(36,7){\makebox(0,0){$0$}}

\end{picture}
\caption{A height function $h$ on $T^{3+1}$.}
\end{figure}

Let $T_1,T_2, \cdots ,T_{d}$ denote $d$ copies of $T^{q+1}$, each with a fixed base point, and let $h_i:T_i \rightarrow \R$ be a height function on $T_i$.  The Diestel-Leader complex is the subset of the product of these trees defined by
$$DL_d(q)= \left\{ (x_1, \ldots, x_d) \in T_1 \times T_2 \times \cdots \times T_{d-1}  \mid  \sum_{i=1}^d h_i(x_i) = 0\right\}.$$
We call the one skeleton a Diestel-Leader graph.
To see the structure of this graph, the vertices and edges are specified as follows.
$$Vert(DL_d(q)) = \left\{ (x_1, \ldots, x_d) \in Vert(T_1 \times T_2 \times \cdots \times T_{d-1}) \mid  \sum_{i=1}^d h_i(x_i) = 0\right\}$$
where there is an edge between $(t_1,t_2, \cdots ,t_{d-1}),(s_1,s_2, \cdots ,s_{d-1}) \in Vert(DL_d(q))$ if and only if there are indices $i,j$ so that there is an edge between $t_i$ and $s_i$ in $T_i$, an edge between $t_j$ and $s_j$ in $T_j$, and for all $k \neq i,j$ we have $t_k=s_k$.

Diestel-Leader graphs arise as an answer to the question ``Is any connected, locally finite, vertex transitive graph quasi-isometric to the Cayley graph of a
finitely generated group?" which is often attributed to Woess.  Analogous Diestel-Leader graphs can be defined as subsets of products of trees with differing valences. When $d=2$, Eskin, Fisher and Whyte in \cite{EFW} show that the Diestel-Leader graph that is a subset of $T^{m+1} \times T^{n+1}$ is not quasi-isometric to the Cayley graph of any finitely generated group when $m \neq n$.  Diestel-Leader graphs which are subsets of $d \geq 3$ trees of differing valence are shown not to be Cayley graphs of finitely generated groups in \cite{BNW}.

\subsection{Diestel-Leader groups} Bartholdi, Neuhauser and Woess in \cite{BNW} construct a group of matrices whose Cayley graph with respect to a particular generating set is identified with the $1$-skeleton of the Diestel-Leader complex $DL_d(q)$.   We denote this group $\GGG$ and refer to it as a {\em Diestel-Leader group}.  As the standard lamplighter groups arise when $d=2$ we can view these Diestel-Leader groups as higher rank lamplighter groups.

The construction in \cite{BNW} is valid when $d-1 \leq p$ for all primes $p$ dividing $q$.  In particular, when $d=2$ or $d=3$, all values of $q$ are permissible.  When this condition is not satisfied, it is not known whether $DL_d(q)$ is the Cayley graph of a finitely generated group; the smallest open case is $DL_4(2)$. It is shown in Corollary 4.5 of \cite{BNW} that $\GGG$ is of type $F_{q-1}$ but not type $F_q$, hence if  $d>2$ these groups are finitely presented.

The matrix groups $\GGG$ are constructed as follows.  Let $\mathcal{L}_q$ be a commutative ring of order $q$ with multiplicative
unit 1, and suppose $\mathcal{L}_q$ contains distinct elements $l_1, \dots, l_{d-1}$ such that their pairwise differences are invertible.  Define a ring of polynomials in the formal variables $t$ and $(t+l_i)^{-1}$ for $1 \leq i \leq d-1$ with
finitely many nonzero coefficients lying in ${\mathcal L}_q$:
$${\mathcal R}_d({\mathcal L}_q) = {\mathcal L}_q[t,(t+l_1)^{-1},(t+l_2)^{-1}, \cdots ,(t+l_{d-1})^{-1}].$$

It is proven in \cite{BNW} that the group $\Gamma_d(q)$ of affine matrices of the
form
\begin{equation}\label{eqn:matrix}
\left( \begin{array}{cc} (t+l_1)^{k_1} \cdots (t+l_{d-1})^{k_{d-1}} & P \\ 0 & 1 \end{array} \right), \text{ with }
k_1,k_2, \cdots ,k_{d-1} \in \Z \text{ and }P \in {\mathcal R}_d({\mathcal L}_q)
\end{equation}
has Cayley graph $DL_d(q)$ with respect to the generating set $\Sigma_{d,q}$ consisting of $d$ types of generators:
\begin{itemize}
\item Type $S_i$ for $i=1,2, \cdots d-1$ consists of matrices of the form $\left( \begin{array}{cc} t+l_i & b \\ 0 & 1 \end{array} \right)^{\pm 1}$, $b$ is an element of the coefficient ring ${\mathcal L}_q$.
\item Type $S_d$ consists of matrices of the form $\left( \begin{array}{cc} \frac{t+l_i}{t+l_j} & \frac{-b}{t+l_j} \\ 0 & 1 \end{array} \right)^{\pm 1}$, for $ \ i,j \in \{1,2, \cdots ,d-1\}, \ i \neq j$ and $b \in {\mathcal L}_q$.
\end{itemize}

We refer the reader to \cite{AR,BNW} or \cite{STW} for a detailed description of the correspondence between the elements of $\GGG$ and the vertices of $DL_d(q)$.  Roughly this correspondence is as follows:
\begin{itemize}
\item The vector $(k_1,k_2, \cdots ,k_{d-1})$ of exponents arising from the upper left entry of the matrix $g$ in Equation \eqref{eqn:matrix} determines the heights of the coordinates of the vertex of $DL_d(q)$ corresponding to this matrix.  Namely, $g$ corresponds to a vertex $(t_1,t_2, \cdots ,t_{d-1},t_d) \in DL_d(q)$ where $h_i(t_i) = k_i$ for $i=1,2, \cdots ,d-1$ and $h_d(t_d) = -(k_1+k_2+ \cdots +k_{d-1})$.

    \smallskip

\item The polynomial $P$ in the upper right entry of the matrix in Equation \eqref{eqn:matrix} determines the specific vertex in each tree at the given height.
\end{itemize}
This correspondence allows us to view the variable $t+l_i$ as associated to the tree $T_i$, the $i$-th tree in the product, for $1 \leq i \leq d-1$.  The variable $t^{-1}$ is then associated with $T_d$.

The change in the upper left entry of any matrix representing a group element under multiplication by a generator from $\Sigma_{d,q}$ is clear.  In particular one can easily see that multiplication by a generator yields a vertex which differs in height from the original vertex in two trees: in one the height has been increased by $1$ and in one the height has been decreased by $1$.

Another way we can view these groups is as a semi-direct product.
$$\Gamma_d(q) = {\mathcal R}_d({\mathcal L}_q) \rtimes \Z^{d-1}$$
where $(k_1, \ldots, k_{d-1}) \in \Z^{d-1}$ acts on ${\mathcal R}_d({\mathcal L}_q)$ by
$$(k_1, \ldots, k_{d-1})\cdot P= P(t+l_1)^{k_1} \cdots (t+l_{d-1})^{k_{d-1}}. $$
If we let $\L_q((t))$ denote the ring of Laurent series with coefficients in $\L_q$ then we also have a discrete cocompact embedding
$${\mathcal R}_d({\mathcal L}_q) \to \prod_{i=1}^d\L_q((t))$$
given by the identifications of $\L_q((t+l_i)) \simeq \L_q((t))$ for $i=1, \ldots , d-1$ and $\L_q((t^{-1})) \simeq \L_q((t))$.
This gives a discrete and cocompact embedding of
$$\Gamma_d(q) ={\mathcal R}_d({\mathcal L}_q) \rtimes \Z^{d-1} \hookrightarrow \left(\prod_{i=1}^d\L_q((t))\right) \rtimes \Z^{d-1}.$$

We will use this point of view in Section \ref{proofsec}.

\subsection{Finite index subgroups of $\GGG$.}\label{finindsec} The proof of Theorem \ref{thm:main} relies on the existence of an index $k$ subgroup of $\GGG$.  We describe such a subgroup by considering group elements corresponding to vertices in $DL_d(q)$ for which the height of the first coordinate lies in $k\Z$.  We use the generating set $\Sigma_{d,q}$ given above to define this subgroup, which we denote $\Gk$.

Let $\Gk$ be the subgroup of $\GGG$ containing all matrices of the form
\begin{equation}\label{eqn:indexk}
\left( \begin{array}{cc} (t+l_1)^{ke_1} (t+l_2)^{e_2}\cdots (t+l_{d-1})^{e_{d-1}} & P \\ 0 & 1 \end{array} \right), \text{ with }
e_1,e_2, \cdots ,e_{d-1} \in \Z \text{ and }P \in {\mathcal R}_d({\mathcal L}_q).
\end{equation}
We claim that $\Gk$ is a subgroup of $\GGG$ of index $k$.

\begin{proposition}\label{prop:indexk}
The subgroup $\Gk$ containing the matrices listed above is a finitely generated subgroup of index $k$ in $\GGG$.
\end{proposition}

\begin{proof}
It is clear that the set of matrices of this form is closed under multiplication and hence the $k$ cosets of $\Gk$ are $ \Gk \left( \begin{array}{cc} t+l_1 & 0 \\ 0 & 1 \end{array} \right)^i$ for $i=0,1,2, \cdots ,k-1$.

As $\Gk$ is finite index in $\GGG$ it is clearly finitely generated; to describe the correspondence between $\Gk$ and a particular subgraph of $DL_d(q)$ is is helpful to list the generators of $\Gk$, which are:
\begin{itemize}
\item products of the form
$$\prod_{i=1}^k \left( \begin{array}{cc} t+l_1 & b_i \\ 0 & 1 \end{array} \right)$$ and their inverses, where $b_i \in \L_q$, and
\item  products of the form
$$\prod_{i=1}^{k_1} \left( \begin{array}{cc} \frac{t+l_1}{t+l_{j_i}} & \frac{-b}{t+l_{j_i}} \\ 0 & 1 \end{array} \right)\prod_{i=1}^{k_2} \left( \begin{array}{cc} t+l_1 & b_i \\ 0 & 1 \end{array} \right)$$ and their inverses, where $k_1+k_2=k$, $j_i \neq 1$ and $b_i \in \L_q$.
\item All remaining generators of $\GGG$ which do not involve $t+l_1$.
\end{itemize}
\end{proof}

Note that the construction of $\Gk$ is completely symmetric in the first $d-1$ variables and will produce additional examples of index $k$ subgroups when $l_1$ is replaced by $l_j$. As the assignment of variables to trees is somewhat arbitrary we could permute the variables in other ways to create additional examples of finite index subgroups.

The finite index subgroups $\Gamma^k_d(q)$ we consider are not Diestel-Leader groups; however, their Cayley graphs (with respect to the given generating set) are closely related to the $1$-skeletons of Diestel-Leader complexes. The above description demonstrates that the Cayley graph of $\Gk$ ``sits inside" of the Cayley graph of $\GGG$  -- it contains all vertices of the Cayley graph of $\GGG$ in which the height of the first coordinate is an integral multiple of $k$.  We make this precise by defining a subgraph $DL^k_d(q)$ of $DL_d(q)$ whose vertices are a subset of $Vert(DL_d(q))$ but whose edges are unions of edges from the original graph.

Let $T_1=\bar{T}^{q^{k}+1}$ be the $q^{k}+1$ valent tree whose edges have length $k$; we can view this tree as being constructed by taking every $k$-th level of vertices from our standard tree $T^{q^{k}+1}$. The height function $h_1$ on this tree maps vertices to $k\Z$. Let $T_i=T^{q+1}$ for $i=2, \ldots d$ with height functions $h_i$ as before.
Define
$$Vert(DL^k_d(q)) = \left\{ (x_1, \ldots, x_d)  \mid x_1 \in Vert(\bar{T}^{q^k+1}), \ x_i \in Vert(T^{q+1}),  \ 2 \leq i \leq d, \  \sum_{i=1}^d h_i(x_i) = 0\right\}.$$

Edges in $DL^k_d(q)$ have one of two forms; suppose that $(s_1,s_2, \cdots ,s_d),(t_1,t_2, \cdots ,t_d) \in Vert(DL^k_d(q))$ differ by an edge in $DL^k_d(q)$.  Then either:
\begin{itemize}
\item there are indices $i,j \in \{1,2, \cdots ,d\}$ so that $s_i$ and $t_i$ are connected by a single edge in $T_i$, and $s_j$ and $t_j$ are connected by a single edge in $T_j$.  These edges are edges in $DL_d(q)$ as well.
\item the vertices $s_1$ and $t_1$ are connected by an edge of length $k$ in $T_1$ and there is a collection of indices $i_1,i_2, \cdots i_r$ so that $s_{i_l}$ and $t_{i_l}$ are connected in $T_{i_l}$ by a path of $q_l$ edges, and $q_1+q_2+ \cdots +q_r = k$.  Together this is a compilation of $k$ edges from $DL_d(q)$.
\end{itemize}

The one skeleton of $DL^k_d(q)$ is the Cayley graph of the index $k$ subgroup $\Gamma^k_d(q)$ of $\Gamma_d(q)$.

\begin{figure}[h]
\setlength{\unitlength}{.1in}
\center
\begin{picture}(20,8)(1,.5)
\linethickness{1pt}
\newsavebox{\treeone}
\savebox{\treeone}(-4,8){
\put(-4,1){\line(2,3){4}}
\put(-3,1){\line(1,2){3}}
\put(-2,1){\line(1,3){2}}

\put(-1,1){\line(1,6){1}}
\put(0,1){\line(0,6){6}}
\put(1,1){\line(-1,6){1}}

\put(2,1){\line(-1,3){2}}
\put(3,1){\line(-1,2){3}}
\put(4,1){\line(-2,3){4}}
}

%
%

\newsavebox{\treetwo}
\savebox{\treetwo}(11,8){
\put(11,1){\line(2,3){4}}
\put(12,1){\line(1,3){1}}
\put(13,1){\line(0,1){3}}

\put(14,1){\line(1,3){1}}
\put(15,1){\line(0,6){6}}
\put(16,1){\line(-1,3){1}}

\put(17,1){\line(0,1){3}}
\put(18,1){\line(-1,3){1}}
\put(19,1){\line(-2,3){4}}
}

\put(-4,-1){\usebox{\treeone}}
\put(-16,-1){\usebox{\treetwo}}
\put(-6,-1){\usebox{\treetwo}}
\put(4,-1){\usebox{\treetwo}}

%
%
%
\put(30,4){\vector(1,0){3}}
\put(35,1){\line(0,1){6}}
\put(36,1){\makebox(0,0){$2$}}
\put(36,4){\makebox(0,0){$1$}}
\put(36,7){\makebox(0,0){$0$}}

\end{picture}
\caption{ $DL^2_4(3)$ is a subset of the product of these trees.}
\end{figure}
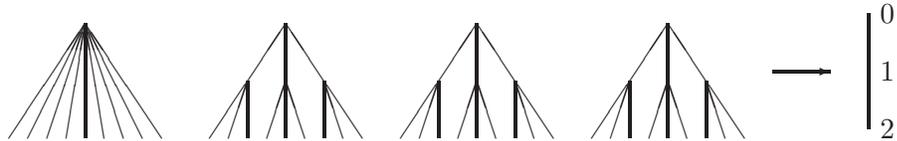

With this description of the geometry of $DL_d^k(q)$ we can give a brief sketch of the proof of Theorem \ref{thm:main}.

{\em Sketch of the proof of Theorem \ref{thm:main}.} First we suppose there is a bijective quasi-isometry $\phi: DL_d^k(q) \to DL_d(q)$. Proposition \ref{ktoone} below then guarantees an induced $k$-to-$1$ quasi-isometry
$$\Phi: DL_d^k(q) \to DL_d^k(q)$$
 (i.e. $|\Phi^{-1}(v)|=k$ for each vertex $v$).  In Section \ref{proofsec} we show that this is impossible if $k$ is not a product of prime factors appearing in $q$.

\section{Boxes and F\o lner sets}\label{boxsec}

In order to define a $k$-to-$1$ map from $DL_d^k(q)$ to itself, we further describe the geometry of the graphs $DL_d^k(q)$.
Note that when $k=1$ we are considering the original Diestel-Leader complex $DL_d(q)$ so this case is covered as well.

First we define a family of sets which we call \emph{boxes}.
Our $k$-to-$1$ map is initially defined on a box of a fixed size. Next we show that $DL_d^k(q)$ can be tiled by a disjoint union of these boxes which allows us to extend the map to the whole of $DL_d^k(q)$. Finally we show that if we take a sequence of boxes of increasing size we obtain a F\o lner sequence (see Definition \ref{Folnerdefn}.)
This is necessary for the results in Section \ref{proofsec}.

{\bf Boxes.} For any $k \in \N$ there is a natural ``height''  map $\rho: Vert(DL^k_d(q)) \to k \Z \times \Z^{d-2}$ given by
$$ \rho(x_1, \ldots, x_d)= (h_1(x_1), \ldots, h_{d-1}(x_{d-1})).$$
A {\em box} will be a connected component of the inverse image of this map. More precisely:

\begin{definition}\label{boxdefn}Let $V_h^k= \prod [a_i, b_i]$ be a subset of $ k\Z \times \Z^{d-2}$ with $|b_i-a_i|=h$ for all $i$.  Define a {\em box} $B_h^k \subset DL_d^k(q)$ to be a connected component of $\rho^{-1}(V_h)$.
\end{definition}

An alternate description of $B_h^k$ is as follows: For each $T_i$ where $i\leq d-1$ take a connected component of $h_i^{-1} [a_i,b_i]$ and for $T_d$  take a connected component of $h_d^{-1}[-\sum b_i,-\sum a_i]$. Then $B_h^k$ is the product of these tree components restricted to $DL_d^k(q)$. To simplify calculations we will assume that $h$ is a multiple of $k$.

\begin{lemma} \label{lemma:counting-vertices} Let $v \in V_h^k$.  Then $\rho^{-1}(v)\cap B_h^k$ contains exactly $q^{(d-1)h}$ vertices.
\end{lemma}

\begin{proof} First note that $B_h^k \cap T_d$ is a subtree of height
$$ -\sum_{i=1}^{d-1} a_i + \sum_{i=1}^{d-1} b_i= (d-1)h$$
and hence contains $q^{(d-1)h}$ vertices at its maximal height.  Thus $\rho^{-1}(a_1,a_2, \cdots ,a_{d-1})\cap B_h^k$ contains $q^{(d-1)h}$ vertices, as each $T_i$ for $1 \leq i \leq d-1$ has a unique vertex in $B_h^k \cap T_i$ at height $a_i$.  We now show that the preimage of any point in $V_h^k$ contains this same number of vertices.

Choose any point $v=(b_1-r_1,b_2-r_2, \cdots ,b_{d-1}-r_{d-1}) \in V_h^k$, where $r_1 \in k\Z$ and $0 \leq r_i \leq h$.  For any $1 \leq i \leq d-1$ the set $\rho^{-1}(v)\cap B_h \cap T_i$ contains
$$q^{b_i-r_i-a_i}$$
vertices.  The height of the $d$-th coordinate of any vertex in $\rho^{-1}(v) \cap B_h^k$ must be $$- \sum_{i=1}^{d-1} b_i + \sum_{i=1}^{d-1} r_i$$ and there are $q^{-\alpha}$ possible vertices in $B_h^k \cap T_d$ at this height, where
$$\alpha = \sum_{i=1}^{d-1} b_i + \sum_{i=1}^{d-1} r_i + \sum_{i=1}^{d-1} b_i = \sum_{i=1}^{d-1} r_i.$$
Thus there are $q^{-\beta}$ vertices in $\rho^{-1}(v)\cap B_h^k$, where
$$\beta=\sum_{i=1}^{d-1} b_i-\sum_{i=1}^{d-1} r_i-\sum_{i=1}^{d-1} a_i +\sum_{i=1}^{d-1} r_i = \sum_{i=1}^{d-1} b_i-\sum_{i=1}^{d-1} a_i = (d-1)h.$$
\end{proof}

We now show that these boxes constitute a F\o lner sequence in $DL_d^k(q)$.

\begin{definition}\label{Folnerdefn}
A F\o lner sequence $F_i$ in a discrete space $X$ is a collection of finite sets with the property that for each $r>0$
$$\lim_{i \rightarrow \infty} \frac{|\partial_r F_i|}{|F_i|} \rightarrow 0.$$
where $\partial_r S$ is the set of all points in $S$ that are distance at most $r$  from  $X\setminus S$.
\end{definition}
The existence of a F\o lner sequence is a defining property of an amenable group.

\begin{lemma} The boxes $B_h$ define a F\o lner sequence. \end{lemma}
\begin{proof}
Since $\partial_r S$ is all points $s \in S$ with $dist(s, X\setminus S) \leq r$ then $\partial_r B_h$ is the set of all $x\in B_h$ with $\rho(x) \in \partial_r V_h$. It follows from Lemma \ref{lemma:counting-vertices} that
$$ |B_h| = |V_h| q^{(d-1)h}$$
and
$$ |\partial_r B_h|= |\partial_r V_h| q^{(d-1)h}$$
Since $V_h$ is a F\o lner sequence in $k\Z \times \Z^{d-2}$ the lemma follows.
\end{proof}

We now use these boxes in $DL_d^k(q)$ to alter a bijective quasi-isometry $\varphi: DL_d^k(q) \to DL_d(q)$ into a $k$-to-$1$ quasi-isometry $$\Phi: DL_d^k(q) \to i(DL_d^k(q)) \simeq DL_d^k(q)$$ which will be used in the proof of Theorem \ref{thm:main}.

\begin{prop}\label{ktoone} Suppose $\varphi: DL_d^k(q) \to DL_d(q)$ is a bijective quasi-isometry. Then by modifying $\varphi$ a bounded amount we get an induced $k$-to-$1$ quasi-isometry $\Phi: DL_d^k(q) \to i(DL_d^k(q)) \simeq DL_d^k(q)$.
\end{prop}
\begin{proof} We will show that there is a $k$-to-$1$ map $u:DL_d(q) \to DL_d^k(q)$ which is a bounded distance from the identity. Then we set $\Phi= u \circ \varphi$.

Note that the vertices of $i(DL_d^k(q))$ are exactly those vertices of $DL_d(q)$ where the height of the $T_1$ coordinate is a multiple of $k$, that is, $\rho(i(DL_d^k(q))) = k\Z \times \Z^{d-2}$.

Let  $a_1$ be a multiple of $k$ and let $h=k$. Lemma \ref{lemma:counting-vertices} ensures that for any $v \in V_h^1$, the set $\rho^{-1}(v)\cap B_h^1 \subset DL_d(q)$ contains a constant number of vertices.
For $v=(x_1, \ldots, x_{d-1}) \in V_h^1$  define
$$\bar{v}=(k\lfloor x_1/k \rfloor, x_2, \ldots, x_{d-1}) \in  k \Z \times \Z^{d-2}$$
and map the vertices in $\rho^{-1}(v)\cap B_h^1$ bijectively to $\rho^{-1}(\bar{v})\cap B_h^1$.
Note also that $DL_d(q)$ can be tiled by boxes that are copies of $B_h^1$ so that we can extend the definition of this map to all of $DL_d(q)$.
This map is $k$-to-$1$ and sends $DL_d(q)$ to $i(DL_d^k(q))$. As vertices in $B_h^1$ are all mapped to vertices in $B_h^1$, all vertices are moved a uniformly bounded amount and hence this map is within a bounded distance of the identity map on this box.
\end{proof}

\section{Boundaries and quasi-isometries}\label{proofsec}

In this section we embed $DL^k_d(q)$ into
$$\prod_{i=1}^d \left(\L_q((t)) \rtimes_\alpha \Z \right)$$
and use this embedding to better understand its quasi-isometries.

\subsection{Laurent series $\L_q((t))$.} In this section we require  only that $\L_q$ is a group of order $q$. Then $\L_q((t))$  denotes the set of Laurent series with coefficients in $\L_q$.  While the construction of the group $\GGG$ requires that $\L_q$ be a ring with specific properties, the result in Theorem \ref{thm:main} is a more general statement for Diestel-Leader graphs, and applies to all $DL_d(q)$ regardless of the choice of parameters, that is, even to those Diestel-Leader graphs which are not the Cayley graphs of finitely generated groups using the construction from \cite{BNW}.

For each $\xi=\sum b_i t^i \in  \L_q((t))$ we define the \emph{clone} of size $q^{-n}$ containing $\xi$ by
$$ C_{\xi,n}= \left\{ \sum a_i t^i  \mid  a_i = b_i \textrm{ for } i \leq n \right\}$$
Note that $\L_q((t))$ has a natural metric space structure where length is given by $|\sum_{i=n}^\infty a_i t^i| = q^{-n}$.  The associated Hausdorff measure will be denoted by $\mu$. Note that
$$\mu(C_{\xi,n})=diam(C_{\xi,n})= q^{-n}.$$

Let $ \alpha: \L_q((t)) \to \L_q((t))$ be the automorphism defined by
$$ \alpha( \sum a_i q^i ) = \sum a_{i-1} q^i.$$
Then $\alpha$ is a \emph{contraction} in that for each $\xi \in \L_q((t))$ we have that $|\alpha^n(C_{\xi,n})| \to 0$ (and $diam(\alpha^n(C_{\xi,n})) \to 0$) as $n \to \infty$.

\subsection{Relation between Laurent series and trees.} The space $\L_q((t)) \rtimes_\alpha \Z$ is roughly isometric (i.e. $(1,C)$-quasi-isometric) to the tree $T^{q+1}$; this is described in more detail, for example, in \cite{CCMT} and \cite{Dy}.

The rough isometry
$$\pi:  \L_q((t)) \rtimes_\alpha \Z \to T^{q+1}$$
is determined in the natural way by the standard identification of $ \L_q((t))$ with the space of vertical geodesics in $T^{q+1}$ (see \cite{Dy} or \cite{EFW} for more details). This identification maps each set of the form $C_{\xi,n} \times \{n\}$ to a single vertex in $T^{q+1}$.   In fact it induces a bijective correspondence between these sets and vertices in the tree.

Similarly, $\pi$ induces a rough isometry
$$ {\pi}_k:  \L_q((t)) \rtimes_\alpha k\Z \to \bar{T}^{q^k+1}.$$

Alternatively, if we rescale the metric on $\Z$ by $k$ then we have a rough isometry
$$ {\pi}_k:  \L_{q^k}((t)) \rtimes_\alpha \Z \to \bar{T}^{q^k+1}.$$
where now $\alpha$ is the standard contraction on $\L_{q^k}((t))$.
In all cases projection to the $\Z$ coordinate corresponds to the height map.

{\bf Remark.} The space of vertical geodesics has been called the
\emph{boundary} of the tree $T^{q+1}$.
In previous literature, $\Q_q$ has been used to denote the boundary of a tree instead of $\L_q((t))$ but here it is more natural to use $\L_q((t))$.
As metric spaces, $\L_q((t))$ and $\Q_q$ are identical.

\subsection{Relation between Laurent series and Diestel-Leader graphs.}
When $d=2$ and $\Gamma_2(q)$ is the lamplighter group $F \wr \Z$ where $|F|=q$, we have cocompact discrete embeddings
$$\Gamma_2(q) \to  (\L_q((t)) \oplus  \L_q((t))) \rtimes \Z$$
where the action of $\Z$ on  $\L_q((t)) \oplus  \L_q((t))$ is given by $(\alpha,\alpha^{-1})$. See, for example, \cite{Shalom}.

In Section 6 of \cite{Dy} the metric on  $\left(\L_q((t)) \oplus  \L_q((t))\right) \rtimes \Z$ is described. The quasi-isometry
$$ \pi: \left(\L_q((t)) \oplus  \L_q((t))\right) \rtimes \Z \to DL_2(q)$$
is also explicitly constructed; to define this map, sets of the form $C_{\xi,n} \times C_{\zeta,-n} \times \{n\}$ are collapsed to vertices of $DL_2(q)$. This correspondence is again a bijection.

Note that we have a (quasi-isometric) embedding

\begin{equation}\label{eqn:semidirect}
\left(\L_q((t)) \oplus  \L_q((t))\right) \rtimes \Z  \hookrightarrow  ( \L_q((t)) \rtimes_\alpha \Z )^2 \simeq T^{q+1} \times T^{q+1}
\end{equation}
where $(\eta, \xi, t)$ is sent to $((\eta, t) , (\xi , -t))$. That is, the sum of the heights of the image points is always zero.

We refer to the two copies of $\L_q((t))$ in Equation \eqref{eqn:semidirect} as the boundaries of $DL_2(q)$.

\subsection{A higher rank analogue of boundary.}\label{higherrankqi}
In analogy to the previous examples we consider the group
$$G=\left( \prod_{i=1}^d  \L_q((t)) \right) \rtimes \Z^{d-1}$$
where $(t_1, \ldots , t_{d-1}) \in \Z^{d-1}$ acts by $( \alpha^{t_1}, \ldots, \alpha^{t_{d-1}}, \alpha^{-(t_1+ \cdots + t_{d-1})})$.

As before we have a quasi-isometric embedding
$$G=\left( \prod_{i=1}^d  \L_q((t)) \right) \rtimes \Z^{d-1} \hookrightarrow
\prod_{i=1}^d \left( \L_q((t)) \rtimes_\alpha \Z\right) \simeq \prod_{i=1}^d T_{q+1}$$
where $$((\xi_1, \ldots , \xi_d), (t_1, \ldots , t_{d-1}))\mapsto ((\xi_1, t_1), \ldots, (\xi_{d-1}, t_{d-1}), (\xi_d, -(t_1+ \cdots + t_{d-1}))).$$
Note that $G$ embeds as the set of all points whose heights sum to zero so that we can identify $G$ with $DL_d(q)$.
The quasi-isometry

$$\pi: \left( \prod_{i=1}^d  \L_q((t)) \right) \rtimes \Z^{d-1} \to  DL_d(q) $$

takes sets of the form $\left(\prod_{i=1}^d C_{\xi_i,t_i}\right) \times \{(t_1, \ldots, t_{d-1})\}$
where  $t_d=-(t_1 + \cdots + t_{d-1})$ and collapses them to vertices of $DL_d(q)$. Again, the correspondence is a bijection between these sets and vertices of $DL_d(q)$. We call the $d$ copies of $\L_q((t))$ the boundaries of $DL_d(q)$.

A natural corollary to our previous statements is that we obtain a quasi-isometry

$$\pi_k: \left( \prod_{i=1}^d  \L_q((t)) \right) \rtimes (k\Z \times \Z^{d-2}) \to  DL_d^k(q) $$

and a bijective correspondence between vertices of $DL_d^k(q)$ and sets of the form $$\left(\prod_{i=1}^d C_{\xi_i,t_i}\right) \times \{(t_1, \ldots, t_{d-1})\}$$ where now $t_1$ is a multiple of $k$.

\subsection{From boundary maps to interior maps.}.

Let $Bilip(\L_q((t)))$ denote all bilipschitz maps $\phi:\L_q((t))\to\L_q((t))$.
Given $d$ bilipschitz maps $\phi_i\in Bilip(\L_q((t)) )$ we can construct a quasi-isometry $\Psi: DL^k_d(q) \to DL^k_d(q)$ by setting
 $$ \Psi = \pi_k \circ \phi_1 \times \cdots \times \phi_d \times id \circ \bar{\pi}_k$$
where $\bar{\pi}_k$ is a coarse inverse of $\pi_k$.

Work of the second author, which generalizes the analogous results for Diestel-Leader graphs $DL_2(q)$ in \cite{EFW}, shows that all quasi-isometries of $DL_d(q)$ are a bounded distance from a function of the above form.  Namely:

\begin{theorem}[Peng]\label{pengthm} Any $(K,C)$ quasi-isometry $\Phi: DL_d(q) \to DL_d(q)$ is bounded distance from a map of the form
$$\pi \circ \phi_1 \times \cdots \phi_d \times id \circ \bar{\pi}$$
where $\phi_i \in Bilip(\L_q((t)))$ and $\pi$ and $\bar{\pi}$ are as above, up to permuting the $\L_q((t))$ factors.
\end{theorem}
\begin{proof} This proof fits into the context  of Peng's work on the structure of quasi-isometries of higher-rank solvable Lie groups. For a brief sketch of this work please see the appendix.
\end{proof}

If $\Phi \in QI(DL^k_d(q))$, we call the maps $\phi_i \in Bilip(\L_q((t)))$ the {\em boundary maps} induced by $\Phi$.  In order to apply certain results from \cite{Dy} we require the boundary maps arising from our bijective quasi-isometry to be particularly nice, that is, \emph{measure linear}.  We first define this property and then  state Proposition \ref{blergh2}, which guarantees two things: first, that we may replace our original quasi-isometry with one whose boundary maps are measure linear, and second that the resulting quasi-isometry is also $k$-to-$1$, with measure linear constants that are products of the prime divisors of $q$.  Any omitted proofs can be found in \cite{Dy}.

\begin{definition} A map $\phi : \L_q((t)) \to \L_q((t))$  is said to be \emph{measure linear}  on $\L_q((t))$ if there exists some $\lambda$ such that for all $A \subset \L_q((t))$
$$\frac{\mu(\phi(A))} {\mu(A)} = \lambda.$$
where $\mu$ is the Hausdorff measure on $\L_q((t))$.
\end{definition}

\begin{proposition}\label{blergh2} Any quasi-isometry $\Phi: DL_d^k(q) \to DL_d^k(q)$ gives rise to a quasi-isometry $\bar{\Phi}$ where the induced boundary maps are measure-linear with measure-linear constants $\lambda_1, \ldots, \lambda_d$.  In addition if $\Phi$ is $k$-to-$1$ then $\bar{\Phi}$ is also $k$-to-$1$, and the $\lambda_i$ are products of prime divisors of $q$.
 \end{proposition}

The proof of Proposition \ref{blergh2} follows from Propositions 4.5 and 4.7 of \cite{Dy}.

The sequence of maps we  consider is the following.
Beginning with our initial bijective quasi-isometry $\phi:DL_d^k(q) \rightarrow DL_d(q)$, we construct a $k$-to-$1$ quasi-isometry $\Phi:DL_d^k(q) \rightarrow DL_d^k(q)$. There is an induced quasi-isometry $\bar{\Phi}:DL_d^k(q) \rightarrow DL_d^k(q)$ where the resulting boundary maps $\phi_i, \cdots ,\phi_d$ are measure linear with constants $\lambda_1,\lambda_2, \cdots , \lambda_d$, where each $\lambda_i$ is a product of powers of prime factors of $q$.  In the next section, using these boundary maps, we construct another quasi-isometry $\Psi: DL_d^k(q) \rightarrow DL_d^k(q)$ which is a bounded distance from $\bar{\Phi}$. Without loss of generality we replace $\Phi$ with $\bar{\Phi}$ for the remainder of this paper.

\subsection{Boxes defined by boundaries}

The identification of $DL_d^k(q)$ and
$$\left(\prod_{i=1}^d \L_q((t))\right) \rtimes (k \Z \times \Z^{d-2})$$
described above allows us to define boxes in $DL_d^k(q)$  in terms of clones in the factors of $$ \L_q((t)) \times \cdots \times  \L_q((t)).$$

\begin{definition}\label{specialset} Given clones $C_i \subset \L_q((t))$ of size $\mu(C_1) = q^{b_i}$  for $i=1, \cdots ,d$  where $b_1$ is a multiple of $k$ set
\begin{itemize}
\item $h=\sum_{i=1}^d b_i$
\item  $a_i= b_i-h$ for $i=1, \ldots, d-1$
\item  $a_d= b_d - (d-1)h$.
\end{itemize}
If $h$ is a nonnegative integer then define
$$V_{b_1, \ldots, b_d}^k= \left\{ (v_1, \ldots, v_{d-1}) \mid  v_i  \in [a_i ,b_i] \textrm{ for } i=1, \ldots,d-1 \right\} \subset k\Z \times \Z^{d-2}.$$
\end{definition}

Note that $V_{b_1, \ldots, b_d}^k=V_h^k$ as in Definition \ref{boxdefn}.  Again, without loss of generality we will always assume that $h$ is a multiple of $k$.

\begin{observation}\label{clonetoboxlemma} Let $C_i \subset \L_q((t))$ be clones with $\mu(C_i)=q^{b_i}$ for $i=1, \ldots, d$ and $b_1$ a multiple of $k$. Let $V_h^k=V_{b_1, \cdots, b_d}^k$ be as in Definition \ref{specialset} above.
Then $$S_{C_1, \ldots , C_d} =\pi_k( C_1 \times \cdots \times C_d \times V_h) \subset DL_d^k(q)$$
is a box $B_h^k$ as in Definition \ref{boxdefn} and for each $v \in V_h$ we have that
$$|\rho^{-1}(v)|=q^{(d-1)h}=q^{b_1}q^{b_2} \ldots q^{b_{d-1}}q^{-\sum a_i} =\mu(C_1)\cdots \mu(C_d).$$
\end{observation}

The observation follows from the definition of $\pi$ and from the fact that $b_i-a_i=h$ for $i=1,2, \cdots ,d-1$ and hence
$$\sum_{i=1}^{d-1} b_i - \sum_{i=1}^{d-1} a_i = \sum_{i=1}^{d-1} (b_i -a_i)=\sum_{i=1}^{d-1} h = (d-1)h.$$

\subsection{Defining the quasi-isometry $\Psi$.}
We now define a quasi-isometry $$\Psi: DL_d^k(q) \rightarrow DL_d^k(q)$$ which is a bounded distance from the map ${\Phi}$ and which
 is on average $ \frac{1}{\lambda_1 \lambda_2 \cdots \lambda_d}$-to-$1$.

\begin{lemma}\label{badonboxes} Let $\Psi: DL^k_d(q) \rightarrow DL^k_d(q)$ be a $(K,C)$-quasi-isometry defined by measure-linear boundary maps $\phi_i$ with measure-linear constants $\lambda_i$ for $i=1,2, \cdots ,d$.
Let $S_{C_1, \ldots , C_d}$ be a box as in Lemma \ref{clonetoboxlemma} with  $h\gg \log_d{K}$. Then, for $r= \log_q{K}$ we have
$$ \frac{1}{\lambda_1 \lambda_2 \cdots \lambda_d} (|  S_{C_1, \ldots , C_d}| - | \partial_r S_{C_1, \ldots , C_d}|) \leq  \sum_{x \in S_{C_1, \ldots , C_d}} | \Psi^{-1}(x) | \leq  \frac{1}{\lambda_1 \lambda_2 \cdots \lambda_d} |  S_{C_1, \ldots , C_d}| + K^d | \partial_r S_{C_1, \ldots , C_d}| $$
\end{lemma}

\begin{proof} As in the proof of Lemma 6.3 in \cite{Dy} we know that $\phi_i^{-1}(C_i)= \sqcup A_{i}^j$
where $A_{i}^j$ are clones of size $$m_i=\mu(A_{i}^j) \geq (1/K) \mu(C_i). $$

In particular there are $\frac{1}{\lambda_im_i}\mu(C_i)$ many clones in $\phi_i^{-1}(C_i)$. This implies that, for a fixed $v \in \Z^{d-1}$, there are
$$\prod_{i=1}^{d} \frac{\mu(C_i)}{\lambda_i m_i}=  \frac{\mu (C_1)\mu(C_2) \cdots \mu(C_d)}{\lambda_1 \ldots \lambda_d} \cdot  \frac{1}{m_1 m_2\cdots m_d}$$
many sets of the form $$ A_{1}^{j_1}\times \cdots A_{d}^{ j_{d}} \times \{v\}$$ in the pre-image of $S_{C_1, \ldots , C_d}$ under $\phi_1 \times \cdots \times \phi_d \times id$. If $v\in V_{\log_qm_1, \ldots, \log_q m_d}$ (see Definition \ref{specialset})
 then by Lemma \ref{clonetoboxlemma} we have that $\pi(  A_{1}^{j_1}\times \cdots A_{d}^{ j_{d}} \times \{v\})$ contains $ \mu(A_{1}^{j_1}) \cdots \mu(A_{d}^{ j_{d}})=m_1m_2\cdots m_d$ many vertices.
  Which means that
$$ \bar{\pi}(DL_d^k(q)) \cap (A_{1}^{j_1}\times \cdots A_{d}^{ j_{d}} \times \{v\})$$ contains $m_1m_2\cdots m_d$ many images of vertices.

Now consider $v\in V_{b_1, \ldots,  b_d}\setminus \partial_r V_{b_1, \ldots,  b_d} $ where $r= \log_q K$.
Then we we have
$$\frac{\mu (C_1)\mu(C_2) \cdots \mu(C_d)}{\lambda_1 \ldots \lambda_d} \cdot  \frac{1}{m_1m_2\cdots m_d} \cdot m_1m_2\cdots m_d = \frac{\mu (C_1)\mu(C_2) \cdots \mu(C_d)}{\lambda_1 \ldots \lambda_d}$$
many vertices being mapped to $S_{C_1, \ldots , C_d}$ at $v \in V$. But by Lemma \ref{clonetoboxlemma} there are $\mu (C_1)\mu(C_2) \cdots \mu(C_d)$ many vertices in $\rho^{-1}(v)$. Therefore there are $1/\lambda_1 \cdots \lambda_d$ many vertices being mapped onto each vertex on average.

Finally for $v \in \partial_r V_{b_1, \ldots,  b_d}$ it is possible, depending on the choice of $\bar{\pi}$, that no vertices or as many as $\prod_{i=1}^d diam(\phi_i^{-1}(C_i))$ are mapped to $\rho^{-1}(v)\cap B_h$. But
$diam(\phi_i^{-1}(C_i)) \leq K \mu(C_i)$ so that there are at most $K^{d} \mu(C_1)\cdots \mu(C_d)$ vertices being mapped to $\rho^{-1}(v)\cap B_h \subset \partial_r B_h.$
\end{proof}

The remainder of the arguments rely on the theory of uniformly finite homology  in spaces of uniformly discrete bounded geometry (denoted $H^{uf}_i(X)$). This is developed in \cite{BW1,BW2} and \cite{W} and an overview is given in \cite{Dy}.  We refer the reader to those references for background.
The main results we will use are the following:

\begin{enumerate}
\item For any such space $X$ there is a fundamental class $[X]$ and if $\chi: X \rightarrow X$ is a $k$-to-$1$ map, then $\chi_*([X]) = k[X]$.
\item Any two quasi-isometries that are a bounded apart induce the same map on homology.
\item For any chain $c=\sum_{x\in X}a_xx \in C_0^{uf}(X)$  we have $[c]=0 \in H_0^{uf}(X)$ if and only if there is some $r>0$ so that for any F\o lner sequence $\{F_i\}$,
\bea
        \left| \sum_{x\in T_i}a_x \right|=O(|\partial_r F_i|).
\eea
\end{enumerate}

\begin{proposition}\label{blergh} If $X=DL_d^k(q)$ and $\phi_1, \ldots, \phi_d$ are measure-linear maps of $\L_q((t))$ with constants $\lambda_1, \ldots, \lambda_d$, define $\Psi = \pi_k \circ (\phi_1 \circ \phi_2 \circ \cdots \circ \phi_d) \circ \bar{\pi}_k \in QI(X)$ as above.  If $k \neq 1/\lambda_1 \lambda_2 \cdots \lambda_d$ then
$$\Psi_*(X)\neq k[X].$$
\end{proposition}
\begin{proof}
The proof of this proposition is similar to the proof of Proposition 6.4 in \cite{Dy}.

Let $S_{C_1, \ldots , C_d}$ be an increasing sequence of boxes in $DL^k_d(q)$, which is necessarily a F\o lner sequence.
Let $c$ be the chain defined by
$$c = \sum_{x\in X} a_x x$$
where $a_x=|\psi^{-1}(x)| -k$.
By Lemma \ref{badonboxes} we have that for $r= \log_q{K}$,
$$ \frac{1}{\lambda_1 \lambda_2 \cdots \lambda_d} (|  S_{C_1, \ldots , C_d}| - | \partial_r S_{C_1, \ldots , C_d}|) \leq  \sum_{x \in S_{C_1, \ldots , C_d}} | \Psi^{-1}(x) | \leq  \frac{1}{\lambda_1 \lambda_2 \cdots \lambda_d} |  S_{C_1, \ldots , C_d}| + K^d | \partial_r S_{C_1, \ldots , C_d}| $$
so that unless $ \frac{1}{\lambda_1 \lambda_2 \cdots \lambda_d}  = k$  we have that $|\sum_{x\in S_{C_1, \ldots , C_d}} a_x |$ is not $O(|\partial_r S_{C_1, \ldots , C_d}|)$ since $|S_{C_1, \ldots , C_d}|$ is not $O(|\partial_r S_{C_1, \ldots , C_d}|)$.
\end{proof}

To conclude the proof of Theorem \ref{thm:main}, note that if $\Psi$ and $\Phi$ have the same boundary maps then they are bounded distance apart. In particular,
$$\Psi_*([X])= \Phi_*([X]).$$ If $\Phi$ is $k$-to-$1$ then $\Phi_*[X]=k[X]$ but if $k$ is not a product of primes appearing in $q$ then by Proposition \ref{blergh2} $k\neq 1/\lambda_1\lambda_2 \cdots \lambda_d$, contradicting Proposition \ref{blergh}.

{\em Proof of Corollary \ref{cor:mainthm}.} Choose $d$ and $q$ so that $d-1 \leq p$ for all primes $p$ dividing $q$, and hence the one skeleton of $DL_d(q)$ is the Cayley graph of the group $\Gamma_d(q)$. Then by Corollary 4.5 of \cite{BNW} $\Gamma_d(q)$ is of type $F_{q-1}$ and hence so is $\Gamma_d^k(q)$ for any $k \in \Z^+$.

For any $k \in \Z^+$, the group $\Gamma_d(q)$ and its index $k$ subgroup $\Gamma_d^k(q)$ are quasi-isometric.  Choose $k \in \Z^+$ which is not a product of prime factors appearing in $q$. Then by Theorem \ref{thm:main} there is no bilipschitz map (that is, bijective quasi-isometry) between $DL_d(q)$ and $DL_d^k(q)$, hence no such map between $\Gamma_d(q)$ and $\Gamma_d^k(q)$ exists.  Thus $\Gamma_d(q)$ and $\Gamma_d^k(q)$ are quasi-isometric but not bilipschitz equivalent. \qed

\section{Appendix}
The goal of this appendix is to place
Theorem \ref{pengthm} (which we restate below with a slightly different perspective) in the context of the work of Eskin, Fisher and Whyte and the second author in \cite{EFW,EFW1,EFW2, P1,P2}.  Namely, using the quasi-isometry between $DL_d(q)$ and $\left( \prod_{i=1}^d  \L_q((t)) \right) \rtimes \Z^{d-1}$ described in Section \ref{higherrankqi}, Theorem \ref{pengthm} can be stated as follows.

{\bf Theorem \ref{pengthm}} (Peng). \emph{Any $(K,C)$ quasi-isometry $$\Phi: \left( \prod_{i=1}^d  \L_q((t)) \right) \rtimes \Z^{d-1} \to \left( \prod_{i=1}^d  \L_q((t)) \right) \rtimes \Z^{d-1}$$ is, up to permuting the $\L_q((t))$ factors, a bounded distance from a map of the form
$$ \phi_1 \times \cdots \phi_d \times id $$
where $\phi_i \in Bilip(\L_q((t)))$.}

We will state the main theorems of Eskin, Fisher and Whyte and the second author but
for a more detailed summary and outline of their work we refer the reader to \cite{EF}. In particular Section 4.4 of \cite{EF} describes the second author's extension of Eskin, Fisher and Whyte's original work.

In \cite{EFW,EFW1,EFW2} Eskin, Fisher and Whyte prove the following theorem.

\begin{theorem}[Eskin-Fisher-Whyte]\label{efwthm} Let $X=Sol$ or $DL_2(q)$. Then any self quasi-isometry of $X$ is, up to permuting the first two coordinates, a bounded distant from a map of the form
$$f_l \times f_u \times id$$
where $f_l, f_u$ are bilipschitz maps of $\R$ if $X=Sol$ or bilipschitz maps of $\L_q((t))$ if $X=DL_2(q)$.
\end{theorem}

Recall that $Sol= \R^2 \rtimes \R$ where the action of $\R$ on $\R^2$ is given by any matrix $A \in SL_2(\R)$.  We can also view $Sol$ as a subset of the product of two hyperbolic planes:  $$Sol =\{((x_1,y_1),(x_2,y_2))\in \mathbb{H}^2 \times \mathbb{H}^2 \mid \ln{y_1}+\ln{y_2}=0\}.$$
This gives $Sol$ an analogous structure to $DL_2(q)$.

The second author proves a broad generalization of Theorem \ref{efwthm} in \cite{P1,P2} that includes many solvable Lie groups of the form $\R^n \rtimes \R^k$. We will only state her theorem in the case where the solvable Lie group has analogous structure to $DL_d(q)$, namely when  $n=d$, $k=d-1$ and the action of $(t_1, \ldots, t_{d-1}) \in \R^{d-1}$ on $\R^n$ is given
by multiplication by the exponential of

$$\bm  t_1 & 0 &  \cdots & 0 \\ 0 & t_2 & 0  & \vdots \\ \vdots& 0 &  \vdots & 0 \\  0 & \cdots&   0 & -(t_1 + \cdots + t_{d-1}) \fm.$$
%

Note that if we restrict the $t_i$ to lie in $\Z$ instead of $\R$ then this matrix  defines the action of $\Z^{d-1}$ on $\left( \prod_{i=1}^d  \L_q((t)) \right)$ that gives the identification of $DL_d(q)$ with $\left( \prod_{i=1}^d  \L_q((t)) \right) \rtimes \Z^{d-1}$.
\begin{theorem}[Peng]\label{pengorigthm}Any self quasi-isometry of $ \R^d \rtimes \R^{d-1}$ is, up to permuting the coordinates of $\R^d$, a bounded distance from a  map of the form
$$f_1 \times \cdots \times f_d\times id$$
where $f_i$ is a bilipschitz map of $\R$. \end{theorem}

The second author's results were not written to include the $DL_d(q)$ case in order to avoid cumbersome notation, but the same dichotomy that enables Eskin, Fisher and Whyte to prove Theorem \ref{efwthm} for $Sol$ and $DL_2(q)$ simultaneously yields a proof of Theorem \ref{pengthm}.

We briefly rework some of the terminology found in \cite{P1,P2} into our context. First define
$$\alpha_i: \Z^{d-1} \to \Z$$ to be the homomorphism that is projection onto the $i$th coordinate for $i=1, \cdots, d-1$  and  set  $$\alpha_d(t_1, \cdots, t_{d-1})=-(t_1+t_2 \cdots + t_{d-1}).$$This ensures that $\sum_i \alpha_i=0$ and gives the action of $\Z^{d-1}$ on $ \prod_{i=1}^d  \L_q((t))$. In the context of the second author's work, the $\alpha_i$ are known as \emph{roots}.

A \emph{flat} is a subset of the form $\left( ( P_i(t) )_{i}, \mathbb{Z}^{d-1} \right)$ where  $ (P(t) )_{i} \in   \prod_{i=1}^d  \L_q((t))$ is fixed. Geodesics that lie in these flats have the form $\left( ( P_i(t) )_{i}, \mathbb{Z} \vec{v} \right)$ where $\vec{v} \in \mathbb{Z}^{d-1}$. The images of these geodesics under a quasi-isometry are the quasi-geodesics to which \emph{coarse differentation} is applied. Loosely speaking, coarse differentiation is the process of finding a scale at which a quasi-geodesic looks approximately like a geodesic.


Note that, if $\vec{v} \in \mathbb{Z}^{d-1}$ is not close to the kernel of $\alpha_{i}$ for any $i$ (which quantitatively means that $|\alpha_{i}(\vec{v})| \geq \delta |\vec{v}|$ for a pre-fixed $\delta > 0$ and a fixed norm on $\mathbb{Z}^{d-1}$) then the subspace
$$ \mathcal{H}_{\vec{v}}=
\{ ( (x_{i})_{i}, \vec{u} )  \mid (x_i)_i \in  \prod_{i=1}^d  \L_q((t)),\  \vec{u} \in \mathbb{Z} \vec{v} \} $$ is quasi-isometric to a Diestel-Leader graph. We can simplify the notation by denoting a point in this Diestel-Leader graph by  $(x_{+},x_{-}, t \vec{v})$ where $t \in \mathbb{Z}$, $x_{+}=(x_{i})_{\alpha_{i}(\vec{v})> 0}$, and $x_{-}=(x_{i})_{\alpha_{i}(\vec{v})< 0}$.
A \emph{quadrilateral} is given by a collection of four geodesic segments of the form
$$ (y_{+}, y_{-},[-L,L] \vec{v}) ,\ (z_{+}, y_{-},[-L,L] \vec{v}),\ (y_{+}, z_{-},[-L,L] \vec{v}),\ (z_{+}, z_{-},[-L,L] \vec{v}) $$ where
$$ (y_{+},p,L_{+}) = (z_{+},p,L_{+}),\ (q,y_{-},L_{-}) = (q, z_{-}, L_{-})\mbox{ for } p \in \{ y_{-}, z_{-} \},\ q \in \{y_{+},z_{+} \} .$$
We refer the reader to Definition 3.1 of \cite{EFW} for the characteristic properties of a quadrilateral.
The basic step of the proofs of Theorems \ref{pengthm} and \ref{pengorigthm} (and of Theorem \ref{efwthm}) is to show that under any quasi-isometry, quadrilaterals are sent to within a bounded distance of quadrilaterals. This is done by applying the theory of coarse differentiation to the images of the geodesic segments defining the quadrilateral.

Finally, for each generic vector $\vec{v}$ we also have a projection from the subspace 
$\mathcal{H}_{\vec{v}}$ to the space  $$\prod_{\{i \mid  \alpha_{i}(\vec{v}) > 0\} }\left(\mathcal{L}_{q}((t)) \right)_i \rtimes \mathbb{Z}$$  where the action of $\mathbb{Z}$ is dictated by the action of $\vec{v}$. (Note that this space is  quasi-isometric to a tree.)
A \emph{block} associated to $\vec{v}$ is just the pre-image of a point under this projection. A large part of the proof is spent analyzing how blocks behave under quasi-isometry and ultimately showing that blocks are mapped to within bounded distance of blocks. In comparison to the proof of Theorem \ref{efwthm}, in the proof of Theorem \ref{pengorigthm} blocks play the same role as horocycles or height level sets in $\mathbb{H}^2$ and $T^{q+1}$. 

The proof of Theorems \ref{pengorigthm} and \ref{pengthm} proceeds by first focusing on a large box, tiled by much smaller boxes; these are the same as the boxes that are defined in Section \ref{boxsec}. The size of the smaller boxes is determined by the coarse differentiation procedure applied to the special geodesics in the large box. 
Then one shows that on a large fraction of most of these smaller boxes a quasi-isometry is bounded distance from a \emph{standard map} (i.e. a map of the form that appears in the conclusions of Theorems \ref{pengthm} and \ref{pengorigthm}).   A priori, the standard map may be different for each smaller box but after analyzing how blocks behave under quasi-isometries one can conclude that the quasi-isometry is sub-linearly close to a single standard map on a large portion of the large box. The final step is to show that this implies that the quasi-isometry is uniformly bounded distance from a standard map on the whole space.

\bibliographystyle{plain}
\bibliography{refs}

\end{document}